\documentclass[a4paper, reqno]{amsart}
\usepackage{amssymb, amsmath, amsthm, bm, latexsym, enumerate, marginnote,  tikz}
\usepackage{float}
\usepackage{cases}
\usepackage[mathcal]{euscript}
\usepackage[colorlinks,
linkcolor=red,
anchorcolor=magenta,
citecolor=blue]{hyperref}
\usepackage{hyperref}
\usepackage{mathrsfs}
\usepackage{amscd}
\usepackage{mathtools}
\usepackage{amsbsy}
\usepackage{graphicx}
\usepackage{color}
\usepackage[all]{xy}
\usepackage{bigstrut}
\usepackage{geometry}
\begin{document}
\newtheorem{theoreme}{Theorem}
\newtheorem{lemma}{Lemma}[section]
\newtheorem{proposition}[lemma]{Proposition}
\newtheorem{corollary}[lemma]{Corollary}
\newtheorem{definition}[lemma]{Definition}
\newtheorem{conjecture}[lemma]{Conjecture}
\newtheorem{remark}[lemma]{Remark}
\newtheorem{exe}{Exercise}
\newtheorem{theorem}[lemma]{Theorem}
\theoremstyle{definition}
\numberwithin{equation}{section}
\newcommand{\ud}{\mathrm{d}}
\newcommand{\R}{\mathbb R}
\newcommand{\TT}{\mathbb T}
\newcommand{\Z}{\mathbb Z}
\newcommand{\N}{\mathbb N}
\newcommand{\Q}{\mathbb Q}
\newcommand{\Var}{\operatorname{Var}}
\newcommand{\tr}{\operatorname{tr}}
\newcommand{\supp}{\operatorname{Supp}}
\newcommand{\intinf}{\int_{-\infty}^\infty}
\newcommand{\me}{\mathrm{e}}
\newcommand{\mi}{\mathrm{i}}
\newcommand{\dif}{\mathrm{d}}
\newcommand{\beq}{\begin{equation}}
\newcommand{\eeq}{\end{equation}}
\newcommand{\beqq}{\begin{equation*}}
\newcommand{\eeqq}{\end{equation*}}
\newcommand{\ben}{\begin{eqnarray}}
\newcommand{\een}{\end{eqnarray}}
\newcommand{\beno}{\begin{eqnarray*}}
\newcommand{\eeno}{\end{eqnarray*}}
\def \f{\frac}
\def\d{\delta}
\def\a{\alpha}
\def\e{\varepsilon}
\def\ld{\lambda}
\def\p{\partial}
\def\v{\varphi}
\newcommand{\D}{\Delta}
\newcommand{\Ld}{\Lambda}
\newcommand{\n}{\nabla}
\newcommand{\GG}{\text{g}}
\begin{abstract}
In this paper, we consider the local smoothing estimate associated with half-wave
 operator $e^{it\sqrt{-\Delta}}$ and the related fractional Schr\"{o}dinger operator.
 Using the broad norm introduced by Guth \cite{GHI},
 we improve the previous best results of local smoothing estimate
 for the wave equation in higher dimensions.
 For the fractional Schr\"{o}dinger operator $e^{it(-\Delta)^{\alpha/2}}$ with
 $\alpha>1$, we establish the sharp local smoothing estimate in term of regularity, which
 improves the previous results of \cite{Guo,RS}.
\end{abstract}

\title[Improved local smoothing estimates for fractional Schr\"odinger operator]
{Improved local smoothing estimates for the fractional Schr\"odinger operator}

\author[C. Gao]{Chuanwei Gao}
\address{Beijing International Center for Mathematical Research, Peking University, Beijing, China}
\email{cwgao@pku.edu.cn}

\author[C. Miao]{Changxing Miao}
\address{Institute for Applied Physics and Computational Mathematics, Beijing, China}
\email{miao\textunderscore changxing@iapcm.ac.cn}

\author[J. Zheng]{Jiqiang Zheng}
\address{Institute for Applied Physics and Computational Mathematics, Beijing, China}
\email{zhengjiqiang@gmail.com, and zheng\textunderscore jiqiang@iapcm.ac.cn}



\begin{abstract}
In this paper, we consider local smoothing estimates for the fractional Schr\"{o}dinger operator $e^{it(-\Delta)^{\alpha/2}}$ with
 $\alpha>1$.
 Using the $k$-broad ``norm" estimates of Guth-Hickman-Iliopoulou \cite{GHI},
 we improve the previously  best-known results of local smoothing estimates
 of \cite{Guo,RS}.
\end{abstract}

\maketitle

\section{introduction}\label{section-1}
Let $u$ be the solution for the Cauchy problem of the fractional Schr\"odinger  equation
\begin{equation}\label{eq-01}\left\{
 \begin{aligned}
     i\partial_tu+(-\Delta)^{\f{\alpha}{2}}u&=0,\quad (t,x)\in\R\times\R^n\\
   u(0,x)&=f(x),
     \end{aligned}\right.
   \end{equation}
   where $\alpha>1$ and $f$ is a Schwartz function. The solution $u$ is expressed by
   \begin{equation}\label{eq:e1}
   u(x,t)=e^{it(-\Delta)^{\f{\alpha}{2}}}f(x):=\f{1}{(2\pi)^n}\int_{\R^n} e^{i(x\cdot\xi+t|\xi|^\alpha)}\hat{f}(\xi)\dif \xi.
   \end{equation}
 We are concerned with $L^p$-regularity for the solution $u$. For the fixed time $t$,   Fefferman and Stein \cite{FS}, Miyachi \cite{Miya2} showed the following optimal  $L^p$ estimate:
 \beq \label{eq-08}
 \|e^{it(-\Delta)^{\f{\alpha}{2}}}f\|_{L^p(\R^n)}\leq C_{t,p}\|f\|_{L^p_{s_{\alpha,p}}(\R^n)},\;\;\;\; s_{\alpha,p}:=\alpha n\Big|\frac{1}{2}-\frac{1}{p}\Big|, \;\; 1<p<\infty,
 \eeq
 where the constant $C_{t,p}$ is locally bounded.

This estimate trivially leads to the following spacetime estimate
\begin{equation}\label{eq-03}
    \Big(\int_1^2 \| e^{it(-\Delta)^{\frac{\alpha}{2}}}f\|_{L^p(\R^{n})}^p \,\ud t \Big)^{1/p} \lesssim \| f \|_{L^p_{s_{\alpha,p}}(\R^n)}.
\end{equation}
As one can see, compared with \eqref{eq-08}, \eqref{eq-03} does not gain any profits from taking an average over time.  In contrast with the fixed time estimate, a natural question appears: can one gain
some regularities by considering the spacetime integral? More precisely, is there an $\varepsilon >0$ such that
\begin{equation}
    \Big(\int_1^2 \| e^{it(-\Delta)^{\frac{\alpha}{2}}}f \|_{L^p(\R^{n})}^p \,\ud t \Big)^{1/p} \lesssim \| f \|_{L^p_{s_{\alpha,p}-\varepsilon}(\R^n)}?
\end{equation}

 Taking the example in \cite{RS} into account, it seems natural to formulate the following local smoothing conjecture for the fractional  Schr\"odinger operator.

\begin{conjecture}[Local smoothing for the fractional Schr\"odinger operator]
Let $\alpha>1,p>2+\frac2n$ and $s\ge\alpha n(\tfrac{1}{2}-\tfrac{1}{p})-\f{\alpha}{p}$. Then
\beq \label{eq-09}
\|e^{it(-\Delta)^{\f{\alpha}{2}}}f\|_{L^p(\R^n\times [1,2])}\leq C_{p,s}\|f\|_{L^p_s(\R^n)}.
\eeq
  \end{conjecture}
When $\alpha=2$, which corresponds to the Schr\"odinger operator, Rogers \cite{Rogers} proposed this conjecture, and showed that it could be deduced from the restriction conjecture. To be more precise, for $q>2+\f{2}{n}$ with $p'=\f{nq}{n+2}$, the adjoint restriction estimate
  \beq\label{eq:res}
  \|e^{it\Delta}f\|_{L^q(\R^{n+1})}\leq \|\hat{f}\|_{L^p(\R^n)}
  \eeq
  will imply that
  \beq \label{eq-10}
\big\|e^{it\Delta}f\big\|_{L^q(\R^n\times [1,2])}\leq C_{q,s}\big\|f\big\|_{L^q_s(\R^n)}, \quad\; s>2 n\big(\tfrac{1}{2}-\tfrac{1}{q}\big)-\tfrac{2}{q}.
\eeq
The proof of the above implication relies deeply on the structure of the phase function and  the ``completing of square" trick.  Roughly speaking,  we may explicitly write $e^{it\Delta}f$ to be
\beq
e^{it\Delta}f(x)=\f{1}{(4\pi i t)^{\f{n}{2}}}\int_{\R^n} e^{i\f{|x-y|^2}{4t}}f(y)\dif y.
\eeq
Squaring out  $|x-y|^2$, we obtain
\beqq
\big|e^{it\Delta}f(x)\big|=\Big|\f{c^{\frac n2}}{t^{\frac n2}}e^{-i\f{c^2\Delta}{t}}\hat{f}\Big(\f{cx}{t}\Big)\Big|.
\eeqq
 This equality and  \eqref{eq:e1} enable us to express $e^{it\Delta}f$ freely in terms of spatial or frequency variables.  After some appropriate reductions and making use of the pseudo-conformal change of variables, one can obtain \eqref{eq-10} by \eqref{eq:res}. The above approach is, unfortunately, unavailable for the general fractional Schr\"odinger operators. Using the bilinear method, Rogers and Seeger \cite{RS} established the sharp local smoothing results for  $p>2+\frac{4}{n+1}$.  Away from the endpoint regularity, their results were further improved by Guo-Roos-Yung in \cite{Guo} by means of the Bourgain-Guth \cite{BG} iteration argument.  In this paper, motivated by the seminal work of Guth \cite{Guth}, up to the endpoint regularity, we further refresh the range of $p$ of  \cite{Guo, RS} by means of  weakened versions of the multilinear restriction estimates of Bennett-Carbery-Tao \cite{BCT06}, the so-called $k$-broad ``norm" estimates.
 \begin{theorem}\label{theoa1}
   Let $\alpha>1$,  $n\ge 1$ and $s>s_{\alpha,p}-\f{\alpha}{p}$ with
   \begin{equation}\label{eq-12} p>\left\{
     \begin{aligned}
    &2\tfrac{3n+4}{3n},\quad \;\text{\rm for $n$ even},\\
     &2\tfrac{3n+5}{3n+1},\quad\; \text{\rm for $n$ odd}.
     \end{aligned}\right.\end{equation}
   Then
   \beq\label{eq-13}
   \big\|e^{it(-\Delta)^{\f{\alpha}{2}}}f\big\|_{L^p(\R^n\times (1,2))}\leq C\|f\|_{L^p_s(\R^n)}.
   \eeq
 \end{theorem}
\begin{remark}
   We recover the sharp local smoothing results for $n=1$  and improve the previously best-known results in \cite{Guo, RS} for $n\geq 3$. In particular, when $n=1,\alpha=2$, the above result follows from the known restriction theorem in $\R^2$ and Rogers's implication.
  \end{remark}

  The crucial observation is that away from the origin, the phase function $|\xi|^\alpha$ with $\alpha>1$ always has non-vanishing Gaussian curvature. This fact facilitates us to incorporate them all into a class of elliptic phase functions. A prototypical example for such class is the Schr\"odinger propagator, for which local smoothing estimates have extensive applications in various aspects. The Strichartz estimate, among other things,  plays a critical role in the study of  the semilinear Schr\"odinger equations.

 The purpose of this paper is to explore to what extent the current available methods and tools for studying the Fourier restriction can be applied to the local smoothing problems. This paper is organized as follows. In Section \ref{section-2}, we will provide some preliminaries and reductions.  In Section \ref{section-3}, we will prove Theorem \ref{theoa1}.  In the appendix, we will show some probable tractable approaches toward further improvement.

 \indent{\bf Notations.} For nonnegative quantities $X$ and $Y$, we will write $X\lesssim Y$ to denote the inequality $X\leq C Y$ for some $C>0$. If $X\lesssim Y\lesssim X$, we will write $X\sim Y$. Dependence of implicit constants on the spatial dimensions or integral exponents such as $p$ will be suppressed; dependence on additional parameters will be indicated by subscripts. For example, $X\lesssim_u Y$ indicates $X\leq CY$ for some $C=C(u)$. We write $A(R)\leq {\rm RapDec}(R) B$ to mean that for any power $\beta$, there is a constant $C_\beta$ such that
 \beqq
| A(R)|\leq C_\beta R^{-\beta}B \,\,\quad\text{for all}\,\; R\geq 1.
 \eeqq

For a spacetime slab $\R^n\times I$, we write $ L_x^rL_t^q(\R^n\times I)$ for the Banach space of functions $u:\R^n\times I\to\mathbb{C}$ equipped with the norm
    $$\|u\|_{L_x^rL_t^q(\R^n\times I)}:=\bigg(\int_{\R^n} \|u(x,\cdot)\|^r_{L_t^q(I)}\dif x\bigg)^{\frac1r},$$
with the usual adjustments when $q$ or $r$ is infinity. When $q=r$, we abbreviate $L_x^qL_t^q=L_{x,t}^q$.
We will also often abbreviate $\|f\|_{L_x^r(\R^n)}$ to $\|f\|_{L^r}$.  For $1\leq r\leq\infty$,
we use $r'$ to denote the dual exponent to $r$ such that $\tfrac{1}{r}+\tfrac{1}{r'}=1$. Throughout the paper, $\chi_E$ is the characteristic function of the set $E$.
We usually denote by $B_r^n(a)$ a ball in $\R^n$ with center $a$ and radius $r$. We will also denote by $B_R^n$ a ball of radius $R$ and arbitrary center in $\R^n$.  Denote by $A(r):=B_{2r}^n(0)\setminus B_{r/2}^n (0)$. We denote $w_{B^{n}_{R}(x_0)}$ to be a nonnegative weight function adapted to the ball $B^{n}_{R}(x_0)$ such that
$$ w_{B^{n}_{R}(x_0)}(x)\lesssim (1+R^{-1}|x-x_0|)^{-M},$$
for some large constant $M\in \mathbb{N}$.

We define the Fourier transform on $\mathbb{R}^n$ by
\begin{equation*}
\aligned \hat{f}(\xi):= \int_{\mathbb{R}^n}e^{-ix\cdot \xi}f(x)\,\dif x:=\mathcal{F}f(\xi).
\endaligned
\end{equation*}
and the inverse Fourier transform by
\beqq
\check{g}(x):=\f{1}{(2\pi)^n}\int_{\R^n} e^{ix\cdot \xi}g(\xi)\dif \xi:=(\mathcal{F}^{-1}g)(x).
\eeqq
These help us to define the fractional differentiation operators $\vert\nabla\vert^s$ and $\langle\nabla\rangle^s$ for $s\in \R$ via
    $$\vert\nabla\vert^s f(x):=\mathcal{F}^{-1}\big\{\vert\xi\vert^s\hat{f}(\xi)\big\}(x)\quad \text{and}\quad \langle\nabla\rangle^s f(x):=
   \mathcal{F}^{-1}\big\{ (1+|\xi|^2)^\frac{s}{2}\hat{f}(\xi)\big\}(x).$$
In this manner, we define the  Sobolev norm of the space $L^p_{\alpha}(\R^n)$  by
    $$\|f\|_{L^p_{\alpha}(\R^n)}:=\big\|\langle\nabla\rangle^\alpha f\big\|_{L^p(\R^n)}.$$

   Let  $\varphi$ be a radial bump function supported
on the ball $|\xi|\leq 2$ and equal to 1 on the ball $|\xi|\leq 1$. For $N\in 2^{\mathbb{Z}}$, we define the Littlewood--Paley projection operators by
\begin{align*}
&\widehat{P_{\leq N}f}(\xi) := \varphi(\xi/N)\widehat{f}(\xi),
 \\ &\widehat{P_{> N}f}(\xi) :=
(1-\varphi(\xi/N))\widehat{f}(\xi),
\\ &\widehat{P_{N}f}(\xi) :=
(\varphi(\xi/N)-\varphi({2\xi}/{N}))\widehat{f}(\xi).
\end{align*}

\section{preliminaries }\label{section-2}

Define the pseudo-differential operator $P$  by
\beqq
Pf(x):=\int_{\R^n}e^{ix\cdot\xi}p(x,\xi)\hat{f}(\xi)\dif \xi,
\eeqq
where the symbol $p(x,\xi)\in C^{\infty}(\R^n\times\R^n)$ satisfies
\beqq
\big|\partial_x^\alpha\partial_{\xi}^\beta  p(x,\xi)\big|\lesssim_{\alpha,\beta}(1+|\xi|)^{-|\beta|},\quad\forall\;\alpha,\;\beta\in \mathbb{N}^n.
\eeqq
It is well known that the pseudo-differential operator $P$ satisfies the following pseudo-locality property
\beq\label{eq:b1}
\int_{|x-x_0|\leq  1}|Pf(x)|^2 \dif x\lesssim_M \int_{\R^n}\frac{|f(x)|^2}{(1+|x-x_0|)^M}\dif x, \quad\; \text{\rm for}\; M\geq 0.
\eeq
One may refer to \cite[Chapter VI]{stein} for details.
Roughly speaking,   the main contribution of $Pf$ in the unit ball about $x_0$ comes from the values of $f(x)$ for  $x$ near that ball, in view of the rapidly decaying term $(1+|x-x_0|)^{-M}$. One may justify \eqref{eq:b1} through integration by parts. In particular, one has
\beqq
\chi_{B_r^n(x_0)}\cdot Pf(x)= P(\chi_{B_{2r}^n(x_0)}f)(x)+ {\rm RapDec}(r)\|f\|_{L^p},
\eeqq
for any $x_0\in\R^n$ and $1< p<\infty.$
 In this section, we will extend the pseudo-locality property to the operator $e^{it\phi(D)}$ given by
  \beqq
  e^{it\phi(D)}f(x):=\int e^{i(x\cdot \xi+t\phi(\xi))}\hat{f}(\xi)\dif \xi,
  \eeqq
  where the function $\phi$ belongs to a class of elliptic phase functions.
\begin{definition}[Elliptic phase functions]
  For a given $n$-tuple consisting of $n$ dyadic numbers $A=(A_1,\cdots,A_n)$, we say the smooth function
   $\phi$ is of  elliptic type ${\rm E}_{A}$, if ${\rm supp}\,\phi\subset B_1^n(0)$ and satisfies the following conditions
  \begin{itemize}
    \item $\phi(0)=0,\nabla\phi(0)=0$.
    \item  Let $0<\lambda_1\leq \cdots\leq \lambda_n $  be the eigenvalues of the Hessian $\big(\frac{\partial^2\phi}{\partial \xi_i\partial \xi_j}\big)_{n\times n}(\xi)$. For all $\xi \in {\rm supp}\,\phi$, $(\lambda_1,\cdots,\lambda_n)\in [A/2,A)$ by which we mean for each $1\leq i\leq n$, $\lambda_i\in [A_i/2,A_i)$.
  \end{itemize}
  \end{definition}

 Let  $\psi$ be a nonnegative smooth function on $\R^n$ such that
  \beq\label{poss}
  {\rm supp}\;\hat{\psi} \subset B_1^n(0),\;\sum_{\ell \in \mathbb{Z}^{n}} \psi(x-\ell)\equiv 1,\quad\; \forall\; x\in \R^n.
   \eeq
  Define  $\psi_{\ell}(x):= \psi(R^{-2}x-\ell)$ and  $f_\ell=\psi_\ell f$.
\begin{lemma}\label{lemma-1}  Assume $\phi\in {\rm E}_{1}=(1,\cdots,1)$ and ${\rm supp}\,\hat{f}\subset B_1^n(0)$.  Then, for any $\varepsilon>0$, there holds
  \beq\label{eq:2}
  |e^{it\phi(D)} f(x)|\lesssim_{\varepsilon}  \big|e^{it\phi(D)} \big(\Psi_{B_{R^{2+\varepsilon}}^n(x_0)}f\big)(x)\big|+{\rm RapDec(R)}\sum_{|\ell|>R^\varepsilon} \big\|f|\psi_\ell(\cdot-x_0)|^{\f{1}{2}}\big\|_{L^p(w_{B_{R^2}^n(x_0)})},
  \eeq
for $(x,t)\in B_{R^{2}}^{n}(x_0)\times [-R^2,R^2]$, $1< p<\infty$, where
\begin{align*}
\Psi_{B_{R^{2+\varepsilon}}^n(x_0)}(x):=&\sum_{|\ell |\leq  R^\varepsilon}\psi(R^{-2}(x-x_0)-\ell).
\end{align*}

\end{lemma}
\begin{proof}
Without loss of generality, we may assume that $x_0=0$.  
 The general cases can be obtained by the following simple observation
\begin{equation*}
e^{it\phi(D)}f(x_0)=(e^{it\phi(D)}f(\cdot+x_0))(0).
\end{equation*}

We rewrite $e^{it\phi(D)} f$ by \eqref{poss}
  \beq\label{equ:eitadd}
 e^{it\phi(D)} f(x)=\sum_{\ell \in \mathbb{Z}^{n}} \int_{\R^n}\int_{\R^n}e^{i((x-y)\cdot \xi+t\phi(\xi))}\eta(\xi)f_\ell (y) \dif\xi\dif y,
  \eeq
  where $\eta(\xi)\in C_c^{\infty}(B_2^n(0))$ with $\eta(\xi)=1$ when $\xi\in B_1^n(0)$. The associated kernel  $K_t(\cdot)$ of the operator $e^{it\phi(D)}\eta(-i\nabla)$ is
  \beqq
  K_t(x)=\int_{\R^n} e^{i(x\cdot\xi+t\phi(\xi))}\eta(\xi)\dif \xi.
  \eeqq
  Note that $|t|\leq R^2$, by stationary phase argument, we obtain
  \beq\label{eq:e2}
  |K_t(x)|\leq C\chi_{|x|\leq CR^{2}}+ C_M\frac{\chi_{|x|\geq CR^{2}}}{(1+|x|)^M}.
  \eeq
   We decompose $e^{it\phi(D)}f(x)$ into two parts
  \begin{align}
  e^{it\phi(D)}f(x)&= \sum_{|\ell |\leq R^\varepsilon} e^{it\phi(D)}f_\ell(x)+\sum_{|\ell|> R^\varepsilon } e^{it\phi(D)}f_\ell(x)\nonumber\\
  &=e^{it\phi(D)}(\Psi_{B_{R^{2+\varepsilon}}^n(0)}f)(x)+\sum_{|\ell|> R^\varepsilon } e^{it\phi(D)}f_\ell(x).\label{eq:5}
  \end{align}
Now we turn to estimate the second term of the right-hand side of  \eqref{eq:5}.  By H\"older's inequality, we have
  \begin{align*}
  \Big|\sum_{|\ell|> R^\varepsilon } e^{it\phi(D)}f_\ell(x)\Big|=&\Big|\sum_{|\ell|>R^\varepsilon} \int_{\R^n} K_t(x-y)f_\ell(y)\dif y\Big|\\
  \le &\Big|\sum_{|\ell|>R^\varepsilon}\int_{\R^n} |K_t(x-y)|^{\f{1}{2}}|\psi_\ell(y)|^{\f{1}{2}}|\psi_\ell(y)|^{\f{1}{2}}|f(y)||K_t(x-y)|^{\f{1}{2}}\dif y\Big|\\
  \le &\sum_{|\ell|>R^\varepsilon}\Big(\int_{\R^n} |K_t(x-y)|^{\f{p'}{2}}|\psi_\ell(y)|^{\f{p'}{2}} \dif y\Big)^{\f{1}{p'}}\Big(\int_{\R^n} |\psi_\ell(y)|^{\f{p}{2}}|f(y)|^p|K_t(x-y)|^{\f{p}{2}} \dif y\Big)^{\f{1}{p}}.
  \end{align*}
  For $(x,t)\in B_{R^2}^{n}(0)\times [-R^2,R^2]$, using the rapidly decaying property of $K_t$ and $\psi$, we have
  \begin{align*}
  |K_t(x-y)\psi_\ell(y)|\lesssim_M\f{R^{-\varepsilon M}}{\big(1+|R^{-2}y-\ell|\big)^M}, \quad\; |\ell|>R^\varepsilon,\; \forall\; x\in B_{R^2}^n(0),\;y\in \R^n,
  \end{align*}
 and
 $$|K_t(x-y)|\lesssim_M  \f{1}{\Big(1+\f{|y|}{R^2}\Big)^{M/2}}, \;\; \forall  \;x\in B_{R^2}^n(0),\;y\in \R^n.$$
 Hence,
 \begin{align*}
   \Big|\sum_{|\ell|> R^\varepsilon } e^{it\phi(D)}f_\ell(x)\Big|\lesssim_M R^{-\varepsilon M+\frac{2n}{p'}}\sum_{|\ell|>R^\varepsilon}
   \big\|f|\psi_\ell|^{\f{1}{2}}\big\|_{L^p(w_{B_{R^2}^n(0)})}.
 \end{align*}
Therefore, we complete the proof.

\end{proof}

As a direct consequence  of Lemma \ref{lemma-1}, we immediately obtain the relation between local and global estimates in the spatial space.
\begin{corollary}\label{lemma-2}
  Let $\phi\in {\rm E}_1$, $s\in \mathbb{R}$, $2< p<\infty$ and $I$ be an interval with $I\subset [-R^2,R^2]$. Suppose that ${\rm supp}\,\hat{f}\subset B_1^n(0)$ and
  \beq\label{eq:7}
  \|e^{it\phi(D)}f\|_{L_{x,t}^p(B_{R^2}^n \times I)}\leq C R^s \|f\|_{L^p},
  \eeq
  then, $\forall \varepsilon>0$, there holds
   \beq\label{eq:3}
  \|e^{it\phi(D)}f\|_{L_{x,t}^p( \R^n\times I )}\lesssim_\varepsilon R^{s+\varepsilon} \|f\|_{L^p}.
  \eeq
\end{corollary}

\begin{proof}
  Let $\{B_{R^2}^n(x_k)\}_{k\in \mathbb{Z}^n}$ with $x_k=kR^2$ be a family of finitely overlapping balls which cover $\R^n$. Thus
  \beqq
   \big\|e^{it\phi(D)}f\big\|_{L_{x,t}^p(\R^n\times I)}^p \leq \sum_k  \big\|e^{it\phi(D)}f\big\|_{L_{x,t}^p(B_{R^2}^n(x_k) \times I)} ^p.
  \eeqq
Using  Lemma \ref{lemma-1}, we get
  \begin{align*}\label{eq:6}
  \big\|e^{it\phi(D)}f\big\|_{L_{x,t}^p ( B_{R^2}^n(x_k) \times I)}\lesssim_\varepsilon & \big\|e^{it\phi(D)}\big(\Psi_{B_{R^{2+\varepsilon/10n}}^n(x_k)}f\big)\big\|_{L_{x,t}^p (B_{R^2}^n(x_k) \times I)}\\
  &+{\rm RapDec}(R)\sum_{|\ell|>R^\varepsilon} \|f|\psi_\ell(\cdot-x_k)|^{\f{1}{2}}\|_{L^p(w_{B_{R^2}^n(x_k)})}.
  \end{align*}

We take the summation with respect to $k$, and obtain
  \begin{align}
\sum_k  \|e^{it\phi(D)}f\|_{L_{x,t}^p (B_{R^2}^n(x_k) \times I)}^p\lesssim_\varepsilon&  \sum_k \Big(\big\|e^{it\phi(D)}\big(\Psi_{B_{R^{2+\varepsilon/10n}}^n(x_k)}f\big)\big\|_{L_{x,t}^p (B_{R^2}^n(x_k)\times I)}\Big)^p\\
&+{\rm RapDec}(R)\sum_k \Big(\sum_{|\ell|>R^\varepsilon} \big\|f|\psi_\ell(\cdot-x_k)|^{\f{1}{2}}\big\|_{L^p(w_{B_{R^2}^n(x_k)})}\Big)^p.
  \label{eq:61}\end{align}
   It follows from \eqref{eq:7}\footnote{It should be noted that the Fourier support condition of $f$ may not be satisfied. However, it can be easily fixed by dividing $B_1^n(0)$ into several smaller balls and considering the estimate on each of these smaller balls.} and the bounded overlapping property of the balls $\{B_{R^2}^n(x_k)\}_k$ that
   \beqq
   \sum_k \Big(\big\|e^{it\phi(D)}\big(\Psi_{B_{R^{2+\varepsilon/10n}}^n(x_k)}f\big)\big\|_{L_{x,t}^p(B_{R^2}^n(x_k)\times I)}\Big)^p\lesssim R^{sp+\varepsilon p}\|f\|_p^p.
   \eeqq
  As for the error term, using Minkowski's inequality and the separation property of $x_k$, we obtain \begin{align*}
   &{\rm RapDec}(R)\sum_k \Big(\sum_{|\ell|>R^\varepsilon} \|f|\psi_\ell(\cdot-x_k)|^{\f{1}{2}}\|_{L^p(w_{B_{R^2}^n(x_k)})}\Big)^p\\
    \lesssim&{\rm RapDec}(R) \|f\|_{L^p}^p.
   \end{align*}
Thus we complete the proof of Corollary \ref{lemma-2}.
\end{proof}

For later use, we also need the following lemma concerning the eigenvalues of the Hessian matrix of a radially symmetric function. One can
 carry out the approach in \cite{CO} to obtain the following Lemma.
\begin{lemma}\label{le-3}
Let $\phi = \phi(|x|) $ be a radially symmetric $C^2$ function on $R^n \backslash \{0\}, n \geq 2$. Then the determinant of the Hessian matrix is
\beqq
{\rm det}\Big( \frac{\partial^2 \phi}{\partial x_i\partial x_j}\Big)_{n\times n}=\Big(\frac{\phi'(r)}{r}\Big)^{n-1}\phi''(r).
\eeqq
Furthermore, the eigenvalues of the Hessian matrix are
$$\underbrace{\frac{\phi'(r)}{r}}_{(n-1)-\text{fold}}, \quad \phi''(r).$$
\end{lemma}

\section{Proof of Theorem \ref{theoa1}}\label{section-3}
This section is devoted to the proof of Theorem \ref{theoa1}.  Bourgain-Guth \cite{BG} have developed a strategy to convert $k-$linear into linear inequalities  in the context of the Fourier extension operators. In \cite{Guth}, Guth observed that full power of the $k-$linear inequality could  be replaced by a certain weakened version of the multilinear estimate for the Fourier extension operators known as $k-$broad ``norm" estimates. Following the  approach developed  by Guth in \cite{Guth}, we shall divide  $e^{it(-\Delta)^{\alpha/2}}f$
 into narrow and broad parts in the frequency space, and one part is around a neighborhood of $(k-1)$-dimensional subspace,
 another comes from its outside.
We estimate the contribution of the first part through the decoupling theorem and an induction on scales argument,  and then use the $k$-broad ``norm" estimates to handle the broad part. In this  process, we should take advantage of the pseudo-locality property of the fractional Schr\"{o}dinger operators.
It is worth noting that, for $\alpha>1$,  we can incorporate the phase functions $|\xi|^{\alpha}$ into the class of  ``elliptic phase functions".

In order to prove  Theorem \ref{theoa1}, we will use the following $k$-broad ``norm" estimates of Guth-Hickman-Iliopoulou \cite{GHI}. For some $0<\varepsilon\ll 1$, let $1\ll K\ll R^\varepsilon$. We assume that $\theta,\tau$ are balls in $\R^n$ of radius $R^{-1}$ and $K^{-1}$, respectively. Correspondingly, we define $G(\theta)$ and $G(\tau)$ to be the set of unit normal vectors as follows:
  $$G(\theta):=\bigg\{\frac{1}{\sqrt{1+|\nabla \phi|^2}}(-\nabla\phi(\xi),1): \xi \in \theta\bigg\},
\quad\;G(\tau):=\bigcup_{\theta\subset \tau}G(\theta).$$
  Let $V\subset \R^{n+1}$ be a $(k-1)$-dimensional subspace.  We denote by  ${\rm Ang}(G(\tau), V)$  the smallest angle between the non-zero vectors $v\in V$ and $v'\in G(\tau)$.

 Define $$ f_{\tau}:=\mathcal F^{-1}(\hat{f}\chi_{\tau}).$$
For each ball $B_{K^2}^{n+1}\subset B_{R^2}^{n}\times [-R^2,R^2]$, define
\beqq
\mu_{\phi}(B_{K^2}^{n+1}):=\min \limits_{V_1,\ldots, V_L}\max\limits_{\tau \notin V_\ell }\Big( \int_{B_{K^2}^{n+1}} |e^{it\phi(D)}f_\tau|^p\dif x\dif t\Big),
\eeqq
where $\tau \notin V_\ell$ means that for all $1\leq \ell \leq L$, ${\rm Ang}(G(\tau),V_{\ell})>K^{-1}$.

Let  $\{B_{K^2}^{n+1}\}$ be a collection of finitely overlapping balls which form a cover of $B_{R^2}^n\times [-R^2,R^2]$. In this setting, we define the $k$-broad ``norm" by
\beqq
\big\|e^{it\phi(D)}f\big\|_{{\rm BL}_{k,L}^p(B_{R^2}^n\times [-R^2,R^2])}^p:=\sum_{B^{n+1}_{K^2}\subset B_{R^2}^n\times [-R^2,R^2]} \mu_{\phi}(B^{n+1}_{K^2}).
\eeqq
\begin{theorem}[\cite{GHI}]\label{theo5}
  Let $2\leq k\leq n+1$ and $\phi\in {\rm E}_A$. There exists a large constant $L$ such that
  \beq\label{eq-22}
  \big\|e^{it\phi(D)}f\big\|_{{\rm BL}_{k,L}^p(B^{n}_R\times [-R,R])}\lesssim_{A,\varepsilon,L}R^\varepsilon \|\hat{f}\|_{L^2(B^{n}_1(0))}, \;{\rm supp}\,\hat{f}\subset B_1^{n}(0)
  \eeq
  for all $\varepsilon>0$ and $p\geq 2(n+k+1)/(n+k-1)$.
\end{theorem}
As a direct consequence of Theorem \ref{theo5},  we obtain
\begin{corollary}\label{cor1}
  Let $2\leq k\leq n+1$, $\varepsilon>0$ and $\phi\in {\rm E}_A$. There is a large constant $L$ such that
  \beq\label{eq-23}
  \|e^{it\phi(D)}f\|_{{\rm BL}_{k,L}^p(B^{n}_{R^2}\times [-R^2,R^2])}\lesssim_{A, \varepsilon,L}R^{2 n(\f{1}{2}-\f{1}{p})+\varepsilon} \|f\|_{L^p(\R^n)},
  \eeq
  for  all $f$  with  ${\rm supp}\;\hat{f}\subset B_1^{n}(0)$ and $p\geq  2(n+k+1)/(n+k-1)$.
\end{corollary}
\begin{proof}
It follows from \eqref{eq:2}  that for $(x,t)\in B^{n}_{R^2}\times [-R^2,R^2]$
  \begin{align*}
  e^{it\phi(D)}f(x)&=\int_{\R^n}e^{i(x\cdot \xi+t\phi(\xi))}\eta(\xi)\hat{f}(\xi)\dif \xi\\
  &=\int_{\R^n}e^{i(x\cdot \xi+t\phi(\xi))}\eta(\xi)\widehat{\big(\Psi_{B_{R^{2+\varepsilon/10n}}^n}f\big)}(\xi)\dif \xi+{\rm RapDec}(R)\|f\|_p,
   \end{align*}
   where $\eta(\xi)$ is the same as in \eqref{equ:eitadd}.
   By Theorem \ref{theo5} and  H\"older's inequality, we  obtain the desired estimate.
\end{proof}

We also need the following decoupling theorem due to Bourgain-Demeter \cite{BoDe2015}.
\begin{theorem}[Decoupling theorem]\label{decthe} Let $\phi\in {\rm E}_A$, then
\beq\label{eq-24}
\Big\|\sum_{\tau}e^{it\phi(D)}f_\tau\Big\|_{L^p(B^{n+1}_{K^2})}\lesssim_{A,\delta} K^{n(\frac12-\frac1p)+\delta}\Big(\sum_{\tau}\|e^{it\phi(D)}f_\tau\|_{L^p(w_{B^{n+1}_{K^2}})}^p\Big)^{\frac1p},
\eeq
 for $2\leq p\leq \frac{2(n+2)}{n}$ and $\delta>0$.
\end{theorem}

As a consequence of Theorem \ref{decthe}, we have
\begin{lemma}\label{ndec} Let $\phi\in {\rm E}_A$ and
 $V\subset \R^{n+1}$ be a $(k-1)$-dimensional linear subspace, then
\beqq
\Big\|\sum_{\tau\in V}e^{it\phi(D)}f_\tau\Big\|_{L^p(B^{n+1}_{K^2})}\lesssim_{A,\delta} K^{(k-2)(\frac12-\frac1p)+\delta}\Big(\sum_{\tau\in V}\|e^{it\phi(D)}f_\tau\|_{L^p(w_{B^{n+1}_{K^2}})}^p\Big)^{\frac1p},
\eeqq
 for $2\leq p\leq \frac{2k}{k-2}$ and $\delta>0$.
Here the sum is taken over all caps $\tau$ for which ${\rm Ang}(G(\tau),V)\leq K^{-1}$.
\end{lemma}
For the proof of  Lemma \ref{ndec}, one may refer to \cite{Guth} for the details.

\noindent{\bf Parabolic rescaling}\;
For given  $\phi\in {\rm E}_A$,  we denote $Q_A(R)$ to be the optimal  constant such that
\beq\label{eq-20}
\|e^{it\phi(D)}f\|_{L^p_{x,t}(B_{R^2}^{n}\times [-R^2,R^2])}\leq Q_A(R)R^{2n(\frac{1}{2}-\frac{1}{p})}\|f\|_{L^p(\R^n)},\quad\;
\text{\rm supp} \,\hat{f}\subset B_{1}^n (0).
\eeq

The parabolic rescaling  transformation establishes the bridge among the estimates at different scales,  which enables us to use an induction on scales argument. We utilize the pseudo-locality property of the propagator $e^{it\phi(D)}$ to establish a parabolic rescaling in  our setting.
\begin{lemma}[Parabolic rescaling]\label{pra}
 Suppose that $\phi\in {\rm E}_1$, and $\tau\subset  \R^n$ is a ball with radius $K^{-1}$, then for $ 0<\varepsilon\ll 1$, we have
  \begin{equation}\label{add-03}
  \big\|e^{it\phi(D)}f_\tau\big\|_{L^p(B^{n}_{R^2}\times [-R^2,R^2])}\leq C(\varepsilon) K^{-2n(\frac{1}{2}-\frac{1}{p})+\frac{2}{p}-\varepsilon }Q_1\Big(\frac{R}{K}\Big)R^{2n(\frac{1}{2}-\frac{1}{p})+\varepsilon}\|f_\tau\|_{L^p}+{\rm RapDec(R)}\|f\|_{L^p}.
  \end{equation}
\end{lemma}
\begin{proof}
  Suppose that ${\rm supp}\;\hat{f}_{\tau}\subset B^{n}_{K^{-1}}(\xi_\tau)$, and denote $\tilde{\phi}(\xi)$ by
\beqq
\tilde{\phi}(\xi):=\phi(\xi)-\phi(\xi_\tau)-\nabla \phi(\xi_\tau)\cdot\xi,
\eeqq
then
  \begin{align}
  \big|e^{it\phi(D)}f_\tau(x)\big|&=\Big|\int_{ B^{n}_{K^{-1}}(\xi_\tau)} e^{i(x\cdot\xi+t\phi(\xi))}\hat{f}_\tau(\xi)\dif \xi\Big|\nonumber\\
  &=\Big|\int_{ B^{n}_{K^{-1}}(\xi_\tau)} e^{i((x+t\nabla\phi(\xi_\tau))\cdot\xi+t\tilde{\phi}(\xi))}\hat{f}_\tau(\xi)\dif \xi \Big|.\label{eq:pro}
\end{align}
Under an invertible map  $\Phi: (x,t)\rightarrow (y, s)$, i.e.
\beqq
x+t\nabla\phi(\xi_\tau)\rightarrow y, \;\; t\rightarrow s,
\eeqq
\eqref{eq:pro} can be reduced to dealing with
\beq\label{eq:pro1}
\Big|\int_{ B^{n}_{K^{-1}}(\xi_\tau)} e^{i(y\cdot\xi+s\tilde{\phi}(\xi))}\hat{f}_\tau(\xi)\dif \xi \Big|.
\eeq
By the change of variables
\beqq
\xi\rightarrow K^{-1}\xi+\xi_{\tau},
\eeqq
\eqref{eq:pro1} is further reduced to estimating
\beqq
\Big|K^{-n}\int_{B_1^n(0)} e^{i(K^{-1}y\cdot\xi+K^{-2}s\tilde{\phi}_\tau(\xi))}\hat{\tilde{f}}_\tau(\xi)\dif \xi \Big|,
\eeqq
where $\tilde{f}_\tau(\cdot)=e^{-iK\xi_\tau \cdot}K^nf_\tau(K\cdot)$ and
\begin{align*}
  \tilde{\phi}_\tau(\xi)=K^2\big(\phi(\xi_\tau+K^{-1}\xi)-\phi(\xi_\tau)-K^{-1}\nabla\phi(\xi_\tau)\cdot \xi\big).
  \end{align*}
Thus, we have
\begin{align}\label{add-02}
\|e^{it\phi(D)}f_\tau\|_{L^p(B^{n}_{R^2}\times[-R^2,R^2])}^p\lesssim K^{-np}
\|(e^{iK^{-2}s\tilde{\phi}_\tau(D)}\tilde{f}_\tau)(K^{-1}\cdot)\|_{L^p(\Phi(B^{n}_{R^2}\times [-R^2,R^2]))}^p.\end{align}
After the change of variables: $y\rightarrow K \tilde x$, $s\rightarrow K^2\tilde{t}$,  we denote  by $\tilde{\Phi}(B^{n}_{R^2}\times [-R^2,R^2])$ the transformed region from $\Phi(B^{n}_{R^2}\times [-R^2,R^2])$.
 Note that $\tilde{\Phi}(B^{n}_{R^2}\times [-R^2,R^2])$ can be contained in a cylinder of the type $B^{n}_{CR^2/K}\times [-CR^2/K^2, CR^2/K^2]$. Thus we may construct a class of cylinders $\{B_\gamma\}_\gamma$  such that
\beqq
\tilde{\Phi}(B^{n}_{R^2}\times [-R^2,R^2])\subset \bigcup_\gamma B_\gamma, \; B_\gamma=: B^{n}_{CR^2/K^2}(c_\gamma)\times [-CR^2/K^2, CR^2/K^2].
\eeqq
Define
\beqq
\tilde{f}_{\gamma,\tau}:= \Psi_{B_{(R/K)^{2+\varepsilon/10n}}^n(c_\gamma)}\tilde{f}_\tau,
\eeqq
where $\Psi$ is the function introduced in Lemma \ref{lemma-1}.

In order to perform the induction on scales argument, we should verify that the phase function $\tilde{\phi}_\tau$ belongs to the elliptic class ${\rm E}_{1}$.
  Obviously, $\tilde{\phi}_\tau(0)=0$, and the Hessian of $\tilde{\phi}_\tau$ is
  \beqq
  \Big(\frac{\partial^2\phi}{\partial \xi_i\partial \xi_j}\Big)_{n\times n}(K^{-1}\xi+\xi_\tau).
  \eeqq
 Noting that  $\phi \in {\rm E}_1$ and $K^{-1}\xi+\xi_\tau \in {\rm supp}\,\phi $ for $\xi\in B_1^n(0)$, we have $\tilde{\phi}_\tau \in {\rm E}_1$. Hence, using Lemma \ref{lemma-1}, we have
\begin{align*}
&K^{-np}\big\|e^{iK^{-2}s\tilde{\phi}_\tau(D)}\tilde{f}_\tau(K^{-1}\cdot)\big\|_{L^p(\Phi(B^{n}_{R^2}\times [-R^2,R^2]))}^p\\
\lesssim& K^{(-n+\frac{2+n}{p})p} \sum_\gamma \|e^{i\tilde t\tilde{\phi}_\tau(D)}\tilde{f}_\tau\|_{L^p(B_\gamma)}^p\\
\lesssim_\varepsilon &K^{(-n+\frac{2+n}{p})p} \sum_\gamma \|e^{i\tilde t\tilde{\phi}_\tau(D)}\tilde{f}_{\gamma,\tau}\|_{L^p(B_\gamma)}^p+{\rm RapDec(R)}\|f\|_{L^p}^p\\
\lesssim_\varepsilon& K^{(-n+\frac{2+n}{p})p}\Big(\frac{R}{K}\Big)^{2np(\frac{1}{2}-\frac{1}{p})}Q^p_1\big(\tfrac{R}{K}\big)
\sum_\gamma \|\tilde f_{\gamma,\tau}\|_{L^p(\R^n)}^p+{\rm RapDec(R)}\|f\|_{L^p}^p\\
\lesssim_\varepsilon & K^{-2np(\frac{1}{2}-\frac{1}{p})+2-\varepsilon}R^{2np(\f{1}{2}-\f{1}{p})+\varepsilon}Q^p_1\big(\tfrac{R}{K}\big)\|f_\tau\|_{L^p}^p +{\rm RapDec(R)}\|f\|_{L^p}^p.
\end{align*}
This inequality together with \eqref{add-02} yields  \eqref{add-03}.

\end{proof}

We come back to prove Theorem \ref{theoa1}.  We first claim that  \eqref{eq-13} can be reduced to showing for $R \geq  1$,
\beq\label{eq:mainn}
\|e^{it(-\Delta)^{\f{\alpha}{2}}} f\|_{L_{x,t}^p(B^{n}_{R^2}\times [-R^2,R^2])}\lesssim_{\alpha,\varepsilon} R^{2 n(\f{1}{2}-\f{1}{p})+\varepsilon}\|f\|_{L^p(\R^n)},\quad \; {\rm supp} ~\hat{f}\subseteq
B_1^n(0),
\eeq
where $p$ is as in \eqref{eq-12}.

Indeed, by Littlewood-Paley decomposition,
$$e^{it(-\Delta)^{\frac{\alpha}{2}}}f=e^{it(-\Delta)^{\frac{\alpha}{2}}}P_{\leq 1}f
+\sum_{N>1} e^{it(-\Delta)^{\frac{\alpha}{2}}} P_N f,$$
and using the fixed-time estimate \eqref{eq-08},
we easily conclude
\beq\label{eq-25}
\big\|e^{it(-\Delta)^{\frac{\alpha}{2}}}P_{\leq 1}f\big\|_{L_{x,t}^p(\R^n\times[1,2])}
\lesssim \|P_{\leq 1}f\|_{L^p_{s_{\alpha,p}}(\R^n)}\lesssim  \|f\|_{L^p(\R^n)}.
\eeq
Now we come to estimate $e^{it(-\Delta)^{\frac{\alpha}{2}}}P_{N}f$ with $N>1$.
For $R\geq 1$, by Corollary \ref{lemma-2}, \eqref{eq:mainn}, we have
 \beqq
\|e^{it(-\Delta)^{\frac{\alpha}{2}}} g\|_{L_{x,t}^p(\R^n\times [\f{1}{2}R^2,R^2])}\lesssim_{\alpha,\varepsilon} R^{2n(\f{1}{2}-\f{1}{p})+\varepsilon}\|g\|_{L^p(\R^n)}, \;\;{\rm supp} \,\hat{g}\subset
{\rm A}(1).
\eeqq
 Therefore, we obtain
 \beqq
\|e^{it(-\Delta)^{\frac{\alpha}{2}}} P_Nf\|_{L_{x,t}^p(\R^n\times [\f{1}{2}R^2/N^\alpha,R^2/N^\alpha])}\lesssim_{\alpha,\varepsilon} R^{2 n(\f{1}{2}-\f{1}{p})+\varepsilon}N^{-\frac{\alpha}{p}}\|P_Nf\|_{L^p(\R^n)}.
\eeqq
Setting $R=\sqrt{2}N^{\alpha/2}$, we have
\beq\label{trian}
\|e^{it(-\Delta)^{\frac{\alpha}{2}}} P_Nf\|_{L_{x,t}^p(\R^n\times [1,2])}\lesssim_{\alpha,\varepsilon} N^{\alpha n(\f{1}{2}-\f{1}{p})-\f{\alpha}{p}+\varepsilon}\|P_Nf\|_{L^p(\R^n)}.
\eeq
This estimate together with \eqref{eq-25} implies that  for $s>s_{\alpha, p}-\tfrac{\alpha}p$
\begin{align*}
\big\|e^{it(-\Delta)^{\frac{\alpha}{2}}}f\big\|_{L_{x,t}^p(\R^n\times[1,2])}&\leq\big\|e^{it(-\Delta)^{\frac{\alpha}{2}}}P_{\leq 1}f\big\|_{L_{x,t}^p(\R^n\times[1,2])}+\sum_{N>1} \|e^{it(-\Delta)^{\frac{\alpha}{2}}} P_Nf\|_{L_{x,t}^p(\R^n\times [1,2])}\\
&\lesssim_{\alpha,\varepsilon}  \|f\|_{L^p(\R^n)}+\sum_{N>1}N^{\alpha n(\f{1}{2}-\f{1}{p})-\f{\alpha}{p}+\varepsilon}\|P_Nf\|_{L^p(\R^n)}\\
&\lesssim_{\alpha,\varepsilon} \|f\|_{L^p_s(\R^n)},
\end{align*}
and so we verify the above claim. This proves Theorem \ref{theoa1} under the assumption of \eqref{eq:mainn}.

 It remains to prove \eqref{eq:mainn}. We first observe that \eqref{eq:mainn} can be deduced from
 \beq\label{eq:main}
\|e^{it\phi(D)} f\|_{L_{x,t}^p(B^{n}_{R^2}\times [-R^2,R^2])}\lesssim_\varepsilon R^{2 n(\f{1}{2}-\f{1}{p})+\varepsilon}\|f\|_{L^p(\R^n)},\quad \; {\rm supp} ~\hat{f}\subseteq
B_1^n(0),
\eeq
where the phase function $\phi(\xi)\in {\rm E}_1$.

Note that  for $\alpha>1,\xi\neq 0$, by Lemma \ref{le-3}, one easily sees that the eigenvalues of the Hessian matrix of $|\xi|^\alpha$ are $$\underbrace{\alpha |\xi|^{\alpha-2}}_{(n-1)-\text{fold}}, \quad \alpha(\alpha-1)|\xi|^{\alpha-2}.$$
Obviously, the phase function $|\xi|^\alpha,\xi\in A(1)$ may not belong to the ${\rm E}_1$. This problem can be fixed by decomposing $A(1)$ into a series of sufficiently small pieces and making appropriate affine transformations.

Now we show that \eqref{eq:main}  can be deduced from the following proposition.
\begin{proposition}\label{pro1}
  Let $\phi \in {\rm E}_1$. Suppose that
  \beqq
  \big\|e^{it\phi(D)}f\big\|_{{\rm BL}_{k,L}^p(B_{R^2}^{n}\times [-R^2,R^2])}\lesssim_{K,\varepsilon}R^{2n(\frac{1}{2}-\frac{1}{p})+\varepsilon}\|f\|_{L^p},
  \eeqq
  for all  $K\geq 1$,  $\varepsilon>0$ and
  $$2\frac{2n-k+4}{2n-k+2}< p\leq \frac{2k}{k-2},$$
  then
  \beqq
  \big\|e^{it\phi(D)}f\big\|_{L^p(B_{R^2}^{n}\times [-R^2,R^2])}\lesssim_{\varepsilon}R^{2n(\frac{1}{2}-\frac{1}{p})+\varepsilon}\|f\|_{L^p}.
  \eeqq
\end{proposition}
{\bf Proof of \eqref{eq:main}}.\; By Corollary \ref{cor1} and Proposition \ref{pro1}, we obtain \eqref{eq:main} if
\beqq
p> \min_{2\leq k\leq n+1}\max\Big\{2\frac{n+k+1}{n+k-1},2\frac{2n-k+4}{2n-k+2} \Big\}.
\eeqq
In particular, if we choose
\begin{align*}
  k=\left\{\begin{aligned}
  &\tfrac{n+3}{2},\quad  n ~\text{\rm is odd},\\
&\tfrac{n+4}{2}, \quad n~ \text{\rm is even},
  \end{aligned}\right.
\end{align*}
then we  obtain the optimal range as in \eqref{eq-12}.

\begin{proof}[{\bf The proof of Proposition \ref{pro1}}]\;
Let $r>0$, for the sake of convenience, we denote $C_{r}^{n+1}$ to be the cylinder $B_r^{n}\times [-r,r]$. For a given ball $B^{n+1}_{K^2}\subset C_{R^2}^{n+1}$, assume that a choice of $(k-1)$-dimensional subspaces $V_1\ldots V_L$  which achieves the minimum in the definition of the $k$-broad ``norm",  we obtain
\beqq
\int_{B^{n+1}_{K^2}} |e^{it\phi(D)}f(x)|^p\dif x\dif t\lesssim K^{O(1)}\max_{\tau\notin V_{\ell}}\int_{B^{n+1}_{K^2}} \big|e^{it\phi(D)}f_\tau(x)\big|^p\dif x\dif t+\sum_{\ell=1}^L\int_{B^{n+1}_{K^2}}\Big|\sum_{\tau\in V_{\ell}} e^{it\phi(D)}f_\tau(x)\Big|^p\dif x\dif t.
\eeqq
Summing over the balls $\{B^{n+1}_{K^2}\}$ yields
\beq\label{eq-26}
\begin{aligned}
\int_{C_{R^2}^{n+1}} \big|e^{it\phi(D)}f(x)\big|^p \dif x \dif t\lesssim&  K^{O(1)}\sum_{B^{n+1}_{K^2}\subset C_{R^2}^{n+1}}\min_{V_1,\ldots V_L}\max_{\tau\notin V_{\ell}}\int_{B^{n+1}_{K^2}} |e^{it\phi(D)}f_\tau(x)|^p\dif x \dif t\\
&+\sum_{B^{n+1}_{K^2}\subset C_{R^2}^{n+1}}\sum_{\ell=1}^L \int_{B^{n+1}_{K^2}}\Big|
\sum_{\tau\in V_{\ell}} e^{it\phi(D)}f_\tau(x)\Big|^p\dif x\dif t.
\end{aligned}
\eeq
 Now we  use  \eqref{eq-23} to estimate the contribution of the first term in the right-hand side  of \eqref{eq-26}.
Let $\varepsilon>0$ be determined lately. Using Corollary \ref{cor1}, we have
\begin{equation}\label{add-04}
\sum_{B^{n+1}_{K^2}\subset C_{R^2}^{n+1}}\min_{V_1,\ldots V_L}\max_{\tau\notin V_{\ell}}\int_{B_{K^2}^{n+1}} |e^{it\phi(D)}f_\tau(x)|^p\dif x\dif t\lesssim C(\varepsilon,L,K) R^{2np(\frac{1}{2}-\frac{1}{p})+ \varepsilon p/2}\|f\|_{L^p}^p.
\end{equation}

Now we use Lemma \ref{ndec} and parabolic rescaling as in Lemma \ref{pra}
to estimate the contribution of the second term in the right-hand side of \eqref{eq-26}.  Let $\delta>0$ to be chosen later.
 It follows from  Lemma \ref{ndec} that
\beqq
\sum_{\ell=1}^L\int_{B^{n+1}_{K^2}}\Big|\sum_{\tau\in V_{\ell}} e^{it\phi(D)}f_\tau(x)\Big|^p\dif x \dif t
\lesssim C(\delta,L) K^\delta K^{(k-2)(\frac 12-\frac1p)p}\sum_\tau\int_{\R^{n+1}} w_{B^{n+1}_{K^2}}\big|e^{it\phi(D)}f_\tau(x)\big|^p\dif x\dif t.
\eeqq
Summing over $B^{n+1}_{K^2}$ in both sides of the above inequality,  we obtain
\begin{align}
&\sum_{B^{n+1}_{K^2}\subset C_{R^2}^{n+1}}\sum_{\ell=1}^L\int_{B^{n+1}_{K^2}}\Big|\sum_{\tau\in V_{\ell}} e^{it\phi(D)}f_\tau(x)\Big|^p\dif x \dif t\nonumber\\ \lesssim& C(\delta,L)  K^\delta K^{(k-2)(\frac{1}{2}-\frac{1}{p})p} \sum_{\tau}\int_{\mathbb{R}^{n+1}} w_{C_{R^2}^{n+1}} \Big|
 e^{it\phi(D)}f_\tau(x)\Big|^p\dif x\dif t.\label{add-05}
\end{align}
Using the rapidly decaying property of the weight function, we have
\beqq
\int_{\mathbb{R}^{n+1}} w_{C_{R^2}^{n+1}} \Big|
 e^{it\phi(D)}f_\tau(x)\Big|^p\dif x\dif t\leq \int_{C_{R^{2+2\delta}}^{n+1}}\Big|
 e^{it\phi(D)}f_\tau(x)\Big|^p\dif x\dif t+{\rm RapDec}(R)\|f\|_{p}^p.
\eeqq
Choosing $\varepsilon_1>0$, we obtain by  Lemma \ref{pra}
\begin{align*}
&\int_{C_{R^{2+2\delta}}^{n+1}}\Big|
 e^{it\phi(D)}f_\tau(x)\Big|^p\dif x\dif t\\
 \lesssim_{\varepsilon_1}& K^{-2 n(\frac{1}{2}-\frac{1}{p})p+2-\varepsilon_1}Q^p_1\Big(\f{R^{1+\delta}}{K}\Big)R^{2(1+ \delta)np(\f{1}{2}-\f{1}{p})+\varepsilon_1}\big\|f_\tau\big\|_p^p+{\rm RapDec}(R)\|f\|_{p}^p.
\end{align*}
Summing over $\tau$ and noting that
\beqq
\sum_{\tau}\|f_\tau\|_p^p\leq C \|f\|_{L^p}^p, \quad\text{ for} \;\;2\leq p\leq \infty,
\eeqq
we have
\begin{align}
&\sum_\tau \int_{\mathbb{R}^{n}} w_{C_{R^2}^{n+1}} \Big|
 e^{it\phi(D)}f_\tau(x)\Big|^p\dif x\dif t\nonumber\\
 \lesssim_{\varepsilon_1}& K^{-2 n(\frac{1}{2}-\frac{1}{p})p+2-\varepsilon_1}Q^p_1\Big(\f{R^{1+ \delta}}{K}\Big)R^{2(1+ \delta)np(\f{1}{2}-\f{1}{p})+\varepsilon_1}\big\|f\big\|_p^p+{\rm RapDec}(R)\|f\|_{p}^p.\label{add-06}
\end{align}
Collecting the estimates \eqref{add-04}-\eqref{add-06} and inserting them into \eqref{eq-26},  we obtain
\begin{align*}
\int_{C_{R^2}^{n+1}}|e^{it\phi(D)}f(x)|^p&\dif x\dif t\leq C(\varepsilon,L,K)R^{2np(\frac{1}{2}-\frac{1}{p})+\varepsilon p/2}\|f\|_{L^p}^p\\&+C(\delta,\varepsilon_1,L)K^\delta R^{2(1+\delta)np(\f{1}{2}-\f{1}{p})+\varepsilon_1}K^{-e(p,k,n)-\varepsilon_1}Q^p_1\Big(\f{R^{1+ \delta}}{K}\Big)\|f\|_{L^p}^p,
\end{align*}
where
$$e(p,k,n):=-(k-2)(\frac{1}{2}-\frac{1}{p})p+2 n(\frac{1}{2}-\frac{1}{p})p-2>0,\quad p>2\frac{2n-k+4}{2 n-k+2}.$$
Therefore by the definition of $Q_1(R)$, we have
\beqq
Q_1^p (R)\leq C(\varepsilon,L,K)R^{\frac{\varepsilon}2p}+C(\delta,\varepsilon_1,L)K^\delta R^{2 \delta np(\f{1}{2}-\f{1}{p})+\varepsilon_1}K^{-e(p,k,n)-\varepsilon_1}Q^p_1\Big(\f{R^{1+\delta}}{K}\Big).
\eeqq

Fix $p>2\frac{2n-k+4}{2 n-k+2}$, let $K=K_0 R^{\tilde \delta}$ with $ K_0>0$ being a large constant to be chosen later, and
$$\delta=\varepsilon_1/2,\;\;\tilde{\delta}=\f{2\varepsilon_1 (1+np(1/2-1/p))}{e(p,k,n)+\varepsilon_1/2},\;\; 0<\varepsilon_1<\f{2\varepsilon e(p,k,n)}{4+4np(\f{1}{2}-\f{1}{p})-\varepsilon}$$  such that the resulting power of $R$ is negative and $0<\tilde \delta <\varepsilon$.

Recall that, in the process of estimating the broad part, the constant $C(\varepsilon,L,K)$ grows at most polynomially with respect to $K$, by choosing $1\ll K_0$  such that $$C(\delta,\varepsilon_1,L)K^{\delta-e(p,k,n)-\varepsilon_1}_0<\f{1}{2},$$ it follows
\beqq
Q_1 (R)\lesssim_\varepsilon R^\varepsilon.
\eeqq
Thus we finish the proof of Proposition \ref{pro1}.
\end{proof}

\section{appendix}\label{section-5}

\noindent{\bf Further tractable approach}.\;
 The result in Theorem \ref{theoa1} relies on the following sharp $k$-broad ``norm" estimates in \cite{GHI}: Let
 $p\geq \frac{2(n+k+1)}{n+k-1},{\rm supp}\hat{f}(\xi)\subset {\rm A}(1)$,
 then
\beq\label{eq-35}
  \|e^{it(-\Delta)^{\frac{\alpha}2}}f\|_{{\rm BL}_{k,L}^p(B^{n}_{R^2}\times [-R^2,R^2])}\lesssim_{\alpha,\varepsilon}R^{\varepsilon} \|\hat{f}\|_{L^2},\;\;\alpha> 1.
  \eeq
  Using the pseudo-locality property of the operator $e^{it(-\Delta)^{\tfrac{\alpha}{2}}}$ with $\alpha> 1$ and H\"older's inequality, we have
  \beq \label{eq-36}\|e^{it(-\Delta)^{\frac{\alpha}2}}f\|_{{\rm BL}_{k,L}^p(B^{n}_{R^2}\times [-R^2,R^2])}\lesssim_{\alpha,\varepsilon}R^{2 n(\f{1}{2}-\f{1}{p})+\varepsilon} \|f\|_{L^p(\R^n)}, \;\;\alpha>1.
  \eeq
  To improve the result in Theorem \ref{theoa1}, we expect to establish \eqref{eq-36} directly for some $p<\frac{2(n+k+1)}{n+k-1}$.
\vskip0.25cm

\noindent{\bf Remark on the local smoothing of the half-wave operator.} There are some troubles in generalizing the above method to handle the local smoothing estimates of the half-wave operator $e^{it\sqrt{-\Delta}}$, which is of great interest. For the restriction problem of the cone operator, using the Lorentz transformation, we can reduce to considering
 \beqq
 \int_{\R^n} e^{i(x_1\xi_1+\cdots +x_n\xi_n+x_{n+1}\frac{\xi^2_1+\cdots +\xi_{n-1}^2}{2\xi_n})}\eta(\xi)f(\xi)\dif \xi,
 \eeqq
of which the structure is well suited for the rescaling argument. But, for the local smoothing problem, we can't establish the corresponding parabolic rescaling lemma in this setting. In fact, the relationship between  $e^{it\sqrt{-\Delta}}f$ and
\beqq
 \int_{\R^n} e^{i(x_1\xi_1+\cdots +x_n\xi_n+x_{n+1}\frac{\xi^2_1+\cdots +\xi_{n-1}^2}{2\xi_n})}\eta(\xi)\hat{f}(\xi)\dif \xi
 \eeqq
 is uncertain for us.

\subsection*{Acknowledgements}  We are grateful for David Beltran's warm comments.
 C.~Gao was supported by Chinese Postdoc Foundation
Grant 8206300279,  C.~Miao was supported by NSFC Grants 12026407 and 11831004, and J.~Zheng was supported by PFCAEP
project No.~YZJJLX2019012 and NSFC Grant 11901041.

\bibliographystyle{amsplain}

\end{document}